\documentclass[a4paper, 10 pt]{elsarticle}
\usepackage{amsmath,amssymb,amsthm}
\usepackage{hyperref}

\journal{Journal of \LaTeX\ Templates}
\nonstopmode \numberwithin{equation}{section}
\newtheorem{theorem}{Theorem}[section]
\newtheorem{lemma}[theorem]{Lemma}

\newtheorem{remark}[theorem]{Remark}

\newtheorem{definition}[theorem]{Definition}

\newcommand{\norm}[1]{\left\|#1\right\|}

\usepackage{xcolor}
\usepackage{tcolorbox}
\usepackage{multicol}
\newtcolorbox[auto counter,number within=chapter]{prb}{%
	colback=green!10!white,colframe=green!20!black,fonttitle=\bfseries}
\begin{document}
	
	\begin{frontmatter}
		
		\title{Existence of Solutions of Nonconvex Multivalued Navier Stokes Equations }
		
		\author{Bholanath Kumbhakar\fnref{myfootnote}}
		\address{Department of Mathematics, Indian Institute of Technology Roorkee}
		\fntext[myfootnote]{Email id: bkumbhakar@mt.iitr.ac.in}
		
		
		\author{Dwijendra Narain Pandey\corref{mycorrespondingauthor}}
		\cortext[mycorrespondingauthor]{Corresponding author}
		\address{Department of Mathematics, Indian Institute of Technology Roorkee}
		\ead{dwij@ma.iitr.ac.in}
		
		
		\begin{abstract}
In this paper, we discuss the existence of local strong solutions for the multivalued version of three-dimensional nonstationary Navier-Stokes equation in Banach spaces. Also, we considered a more general inclusion problem and studied the existence of solutions using the fixed point technique approach. We assume that the multivalued map possesses closed values (not necessarily convex values) and apply the Schauder Fixed Point Theorem in order to deduce the existence of fixed points. 
		\end{abstract}
		
		
	\end{frontmatter}
\section{Introduction}
 Let $D\subset \mathbb{R}^3$ be a bounded, connected, and open set whose boundary $\partial D$ is smooth. Assume $a>0$ is a given real number and set $Q=D\times (0,a), \Sigma=\partial D\times (0,a)$.
 We consider the three-dimensional Navier Stokes equation involving multivalued maps 
 \begin{equation}\label{Navier Stokes}
    \partial_t u-\frac{1}{\operatorname{Re}}\Delta u+(u\cdot \nabla)u+\nabla p\in F(t,u),~\text{in}~D\times (0,a)
 \end{equation}
 \begin{equation}\label{Div Condition}
      \nabla\cdot u=0,~\text{in}~D\times (0,a)
 \end{equation}
 \begin{equation}\label{BC}
       u=0,~\text{on}~\partial D\times(0,a)
 \end{equation}
 \begin{equation}\label{IC}
       u(\cdot,0)=u_0,~\text{in}~D,
 \end{equation}
 where, $u:D\times [0,a]\to \mathbb{R}^3$ represents the unknown velocity field, $p:D\times [0,a]\to \mathbb{R}$ denotes unknown pressure field, $\operatorname{Re}$ is the Reynolds number, describing the relation between the inertial and viscous forces in the fluid. The map $F$ is a given multivalued nonlinearity that is known to some degree of accuracy. The Navier-Stokes Equations (NSE) describe the time evolution of the velocity and
pressure of a viscous incompressible fluid.

In this paper, we establish the existence of a local strong solution to the problem \eqref{Navier Stokes}-\eqref{IC}.

Next, we consider a more general problem, as given below, and obtain strong solutions.

Let $a>0$, $X, Z$ be real, reflexive, and separable Banach spaces, and $Y$ be a real Banach space. Also, we assume the continuous embeddings
\begin{equation}
    X\hookrightarrow Y \hookrightarrow Z,
\end{equation}
and the embedding $X\hookrightarrow Z$ is compact. Further, let $b\in (0,a], ~q\in (1,\infty)$ and let
\begin{equation}
    X(b)=\{u\in L^q([0,b],X): u^{\prime}~\text{exists and}~u^{\prime}\in L^q([0,b],Z)\}, 
\end{equation}
equipped with the norm
\begin{equation}
\norm{u}_{X(b)}=\left(\norm{u}^q_{L^q([0,b],X)}+\norm{u^{\prime}}^q_{L^q([0,b],Z)}\right)^{\frac{1}{q}}.
\end{equation}
Here, $u^{\prime}$ stands for the derivative of $u$ in the distributional sense.
Also, assume the operators $A(t), B(t): X\to Z$ are single-valued and the map $F:[0,b]\times Z\multimap Z$ has closed values (not necessarily convex values). Then we consider the problem:\\
\textbf{Problem:} find $u\in X(b)$ such that
\begin{equation}\label{Inclusion Problem 1}
    \begin{cases}
        u^{\prime}(t)+A(t)u(t)+B(t)u(t)\in F(t,u(t)),~\text{in}~Z,~t\in (0,b).\\
        u(0)=u_0~\text{in}~Z.
    \end{cases}
\end{equation}
The authors in \cite{eikmeier2023existence} studied the problem \eqref{Inclusion Problem 1} under the assumption that the multimap $F$ has a convex value. In this paper, we remove this convexity assumption and obtain strong solutions. 

 For the existence of weak solutions (even for all dimensions) for single-valued Navier-Stokes problems \eqref{Navier Stokes}-\eqref{IC}, we have a quite complete answer: there exists a global weak solution for every general data: for instance, $u_0\in L^2(D)$ and $f\in L^2(D\times (0, a))$. The existence of weak solutions for the (nonstationary) Navier–Stokes equations was first established in the seminal work of Leray \cite{MR1555394} in 1934. Since then, numerous researchers have explored different formulations of the Navier–Stokes equations, including the contributions of Hopf \cite{MR50423}, Ladyzhenskaya \cite{MR108963, MR254401}, and Lions \cite{MR259693}, among others. More generalized models, such as the Oldroyd model, have been investigated in the presence of a set-valued right-hand side. Additionally, topological degree theory has been introduced to study such problems, as seen in the works of Gori et al. \cite{MR2305430}, Obukhovskii, Zecca, and Zvyagin \cite{MR2078195}, and Zvyagin and Kuzmin \cite{MR2336444}.

 In the single-valued case, the existence of strong (or regular) solutions to the three-dimensional nonstationary Navier–Stokes equations has been examined in works such as Ladyzhenskaya \cite[Section 6.4]{MR254401}, Lions \cite[Section 1.6]{MR259693}, Shinbrot and Kaniel \cite{MR192215}, and Solonnikov \cite{MR172014}. For the Oldroyd model, strong solutions to the corresponding multivalued problem were explored in \cite{MR4478379}, where it was assumed that the multivalued map has convex values and the closed graph property. The specific framework they considered in this section was investigated by Giga and Sohr \cite{MR1138838} for the nonstationary Stokes system and by Fernández-Cara, Guillén, and Ortega \cite{MR1893419} for the more general Oldroyd model for viscoelastic fluids, which includes the Navier–Stokes equations as a special case. However, the convexity condition on the multimap is quite restrictive and not commonly satisfied. Moreover, the existence of strong solutions for the three-dimensional Navier–Stokes multivalued problem has not been addressed when the multivalued nonlinearity lacks convex values. Based on this observation, in this paper, we derive the existence of strong solutions for the problem \eqref{Navier Stokes} without the convexity condition on the multimap $F$. 

 Differential inclusions, which emerged as a natural generalization of the concept of ordinary differential equations, have permeated various areas of science due to their numerous applications and constitute a branch of the general theory of differential equations. A comprehensive review of the differential inclusion theory may be found in \cite{MR1485775}.

We emphasized that convexity plays a vital role in the theory of differential inclusion as most of the fixed point theorems are required to provide existence results, and the fixed point map needs to have the condition that its values are convex. In the statement of Michael's Selection Theorem \cite[Theorem 3$^{\prime \prime}$]{MR77107}, the convexity of the set-valued map is also necessary to obtain the existence of continuous selection. Moreover, the convexity of the set-valued map plays a vital role in determining Caratheodory selections \cite[Theorem 2.74, Chapter II]{MR1485775} of a multimap. In summary, the use of convexity makes it comprehend the existence, uniqueness, stability, and optimization properties of solutions of differential inclusion. Therefore, without convexity, the above-mentioned properties of solutions of differential inclusion are quite complicated. 

However, solutions of differential inclusion with nonconvex-valued multimap have been extensively studied in recent years because it has practical applications in optimization theory and optimal control theory \cite{MR3974775}. Moreover, nonconvexity arises in several real-world problems, such as problems arising in pigeon population models \cite{MR3638196},  population models with hysteresis \cite{MR4218398}, carbonation models \cite{MR3894328} etc. In conclusion, the lack of convexity adds complexity to the studied differential inclusions yet enhances their practical applicability, serving as the core motivation for this research endeavor. For more details about differential inclusions, we mention the works \cite{MR2532526},     \cite{MR2345966}, \cite{MR47317} and references cited therein.\\
The paper is organized as follows:
\begin{itemize}
    \item Section 1 is the introduction, where we give the problem statement and related literature review.
    \item In Section 2, we revisit the preliminary results used in this paper.
    \item In Section 3, we derive the existence result for the three-dimensional Navier Stokes equation.
    \item In Section 4, we generalized the result proved in Section 3 to a more general class of problems.
\end{itemize}
 \section{Preliminaries}
 Assume that $D\subset \mathbb{R}^3$ is a bounded domain, that is a bounded, open, connected set with smooth boundary $\partial D$. Then for a differentiable function $u:D\to \mathbb{R}$, the gradient of \(u\) at a point \(x\in D\) is given by 
 \begin{equation}
     \nabla u(x)=\left(\frac{\partial u}{\partial x_i}(x)\right)_{i=1}^{3}.
 \end{equation}
 In the above, $\frac{\partial}{\partial x_i}, ~i=1,2,3$ denote the partial derivatives of the function $u$. If we consider $u:D\to \mathbb{R}^3$ then we define the gradient of $u$ as
 \begin{equation}
     \nabla u(x)=\left(\frac{\partial u_j}{\partial x_i}(x)\right)_{i,j=1}^{3}.
 \end{equation}
 Further, we define the divergence of $u$ at the point $x\in D$ as follows:
 \begin{equation}
     \nabla\cdot u(x)=\sum_{i=1}^{3}\frac{\partial u_i(x)}{\partial x_i}.
 \end{equation}
To study the existence of strong solutions, we first introduce some functional spaces. Following the book \cite{MR1893419}, for $1<p,q<\infty$ we introduce the spaces
\begin{equation}
    H_p=\{v\in L^p(D)^3: \nabla\cdot v=0, ~v\cdot n=0~\text{on}~\partial D\},
\end{equation}
\begin{equation}
    V_p=H_p\cap W_0^{1,p}(D)^3.
\end{equation}
Here, $n=n(x)$ is a unit vector, normal to $\partial D$ at $x$ and oriented towards the exterior of $D$. Also, the condition $\nabla\cdot v=0$ is understood in the distributional sense. Endowed with the norm of $L^p(D)^3$ (respectively $W^{1,p}(D)^3$), $H_p$ (respectively $V_p$) is a reflexive Banach space. 

We now introduce the Helmholtz projector $P_p: L^p(D)^3\mapsto H_p$. It is a bounded linear operator characterized by the equality $P_pv=v_0$, where $v_0$ is given by the so called Helmholtz decomposition $v=v_0+\nabla  \phi$, with $v_0\in H_p$ and $\phi\in W^{1,p}(D)$. See \cite[Chapter III.1]{MR2808162} for more details). The Stokes operator $A_p:D(A_p)\to H_p$ is defined on $D(A_p)=W^{2,p}(D)^3\cap V_p$ and given by $A_pv=P_p(-\Delta v), ~\text{for all}~v\in D(A_p)$. $D(A_p)$ is a Banach space for the norm 
\begin{equation}
    \norm{v}_{D(A_p)}=\norm{v}_{H_p}+\norm{A_pv}_{H_p}.
\end{equation}
Of course, in $D(A_p)$ this norm is equivalent to $\norm{A_pv}_{H_p}$ and also to the usual norm in the Sobolev space $W^{2,p}(D)^3$ (because of the smoothness of $\partial D$). The operator $-A_p$ is the generator of a bounded analytic semigroup of class $C_0$, $\{e^{-tA_p}: t\ge 0\}$ in $H_p$. 
We now introduce the space
\begin{equation}
    D_p^q=\left\{u\in H_p: \int_{0}^{\infty}\norm{A_pe^{-tA_p}u}^q_{H_p}dt<\infty\right\},
\end{equation}
which is equipped with the norm
\begin{equation}
    \norm{u}_{D_p^q}=\norm{u}_{H_p}+\left( \int_{0}^{\infty}\norm{A_pe^{-tA_p}u}^q_{H_p}dt<\infty\right)^{\frac{1}{q}},
\end{equation}
again a Banach space, coinciding with a real interpolation space between $D(A_p)$ and $H_p$ and with the continuous and dense embeddings $D(A_p)\subset D_p^q\subset H_p$, see \cite{MR230022}).

 Now, we can state the following results useful in the next sequel.
 \begin{lemma}\label{main lemma}\cite[Lemma 10.1]{MR1893419}
     Let $D\subset \mathbb{R}^N, (N\ge 2)$ be a bounded open set with $\partial D\in C^{2,\mu} (0<\mu<1)$ and assume $1<p,q<\infty, ~b>0$. If $u_0\in D_p^q, ~g\in L^q([0,b],H_p)$, then there exists a unique function $u$ satisfying $u\in L^q([0,b],D(A_p)), ~\partial_t u\in L^q([0,b],H_p)$ and
     \begin{equation}\label{Stokes equation}
         \begin{cases}
           \partial_t u+\frac{1}{\operatorname{Re}}A_pu=g, ~\text{a.a. in}~(0,b)\\
             u(\cdot,0)=u_0.
         \end{cases}
     \end{equation}
     Furthermore,
     \begin{equation}
        \norm{u}^q_{L^q([0,b],D(A_p))}+\norm{\partial_tu}^q_{L^q([0,b],H_p)}\le c_A\left(\norm{u_0}^q_{D_p^q}+\norm{g}^q_{L^q([0,b],H_p)}\right),
     \end{equation}
     where $c_A=c_A(\operatorname{Re}, p, q, D)$.
 \end{lemma}
 \begin{lemma}\label{estimate lemma}\cite[Lemma 10.2]{MR1893419}
Let $D \subset \mathbb{R}^N$ ($N \geq 2$) be a bounded open set, with $\partial D \in C^1$ and assume $N < p < +\infty$, $1 < q < +\infty$, $b > 0$. Assume that $u \in L^q([0,b],D(A_p))$ and $\partial_t u \in L^q([0,b],H_p)$. Then
\begin{equation*}
    u \in C([0,b],H_p) \cap L^{2q}([0,b],V_p) \cap L^{\frac{2pq}{N}}([0,b],L^\infty(D)^3).
\end{equation*}

Furthermore,
\begin{equation}
    \|u\|_{L^{2q}([0,b],V_p)} \leq C \|u\|_{L^\infty([0,b],H_p)}^{1/2} \|u\|_{L^q([0,b],D(A_p))}^{1/2}
    \tag{10.3}
\end{equation}
and
\begin{equation}
    \|u\|_{L^{\frac{2pq}{N}}([0,b],L^\infty(D)^3)} \leq C \|u\|_{L^\infty([0,b],H_p)}^{1-\frac{N}{p}} \|u\|_{L^{2q}([0,b],V_p)}^{\frac{N}{p}},
    \tag{10.4}
\end{equation}
where $C = C(p, D)$.
\end{lemma}
\subsection{Multivalued Maps}
Let us mention some preliminaries concerning multivalued maps. We refer to the books \cite{MR1485775},\cite{MR2976197} for this. 

The following definitions clarify the continuity and measurability of multivalued maps, with which we start this section.
\begin{definition}
	Let $Y$ be a Banach space. A multivalued map $\Gamma: X\multimap Y$ is considered convex valued if the set $\Gamma(x)$ is convex for every $x\in X$ and closed valued if $\Gamma(x)$ is closed for all $x\in X$.
\end{definition}
\begin{definition}
	A multivalued map $\Gamma: X\multimap Y$ is lower semicontinuous at the point $x_0\in X$ if $\{x_n\}_{n\in \mathbb{N}}$ is a sequence in $X$ such that $x_n\to x_0\in X$, then for every $y\in \Gamma(x_0)$ there exists a sequence $\{y_n\}_{n\in \mathbb{N}}$ in $Y$ with $y_n\in \Gamma(x_n)$ for every $n\in \mathbb{N}$ such that $y_n\to y$.
\end{definition}
Let $\mathcal{B}(X)$ and $\mathcal{B}(Y)$ be the $\sigma$-algebra of Borel sets in $X$ and $Y$ respectively, and let $\Sigma\otimes\mathcal{B}(X)$ be the $\sigma$-algebra of sets in $I\times X$ generated by sets $U\times V$, where $U\in \Sigma, V\in \mathcal{B}(X)$.
\begin{definition}\cite{MR1485775}
	A multivalued mapping $\Gamma: I\times X\multimap Y$ is product measurable if
	\begin{equation}
		\Gamma^{-1}(V)=\{(t,x)\in I\times X: \Gamma(t,x)\cap V\neq \phi\}\in \Sigma\otimes\mathcal{B}(X),
	\end{equation}
	for any closed set $V\subset Y$.
\end{definition}
The Michael Selection Theorem guarantees the existence of a continuous selection for a lower semicontinuous multimap with convex values. However, this Theorem doesn't hold for multimap with nonconvex values. In such cases, the author of \cite{MR730018} introduces the concept of decomposable sets to establish the existence of a continuous selection. Let us define decomposable sets now.
Formally, the concept of decomposability resembles that of convexity, and as we will see in this section, decomposable sets behave like convex sets. For this reason, decomposable sets play a central role in many applications.
\begin{definition}
A set $K\subset L^1(I,X)$ is said to be decomposable if for every triple $(D,f_1,f_2)\in \Sigma\times K\times K$ we have
	\begin{equation}
		\chi_{D} f_1+\chi_{I\setminus D}f_2\in K.
	\end{equation}
\end{definition}
We now consider the nonconvex-decomposable version of Michael's selection theorem.
\begin{theorem}\label{decomposableselectiontheorem}\cite[Theorem 4.5.32]{MR2024162}
If $Z$ is a separable metric space, $X$ is a separable Banach space and $F: Z\multimap L^1(I, X)$ is lower semicontinuous and has closed decomposable values, then $F$ admits a continuous selection.
\end{theorem}
\section{Navier Stokes Equations}
In this section, we study the existence of local strong solutions for the nonstationary multivalued version of Navier-Stokes equations given by \eqref{Navier Stokes}-\eqref{IC}. Throughout this section, we assume that $3<p<\infty$ and $1<q<\infty$. We assume the following conditions on the multimap $F$.

The multimap $F:[0,a]\times H_p\multimap L^p(D)^3$ satisfies the following properties:
\begin{itemize}
 \item[(F1)] $F: [0,a]\times  H_p \multimap L^p(D)^3 $ is product measurable.
    \item[(F2)] $F(t,\cdot): H_p\multimap L^p(D)^3$ is lower semicontinuous for a.a. $t\in [0,a]$.
    \item[(F3)] there exists $\alpha\in L^q(0,a)$ with $\alpha\ge 0$ for a.a. $ t\in (0,a)$ and a monotonically increasing function $\eta_F:[0,\infty)\to [0,\infty)$ such that
    \begin{equation}
        \norm{F(t,u)}_{L^p(D)^3}\le \alpha(t)(1+\eta_F(\norm{u}_{H_p})),
    \end{equation}
    for a.a. $t\in (0,a)$ and all $u\in H_p$.
\end{itemize}
We are now ready to prove the first main result of this paper.
\begin{theorem}\label{Main Theorem}
Let $D \subset \mathbb{R}^3$ be open, bounded and connected with $\partial D\in C^{2, \mu}, ~0<\mu<1$ and $a>0$. Also let $u_0\in D_p^q$ with $3<p<\infty, ~1<q<\infty$ and $F:[0,a]\times H_p\multimap L^p(D)^3$ satisfying the assumptions (F1)-(F3). Then there exists $b>0$
and 
\begin{equation}
    u\in L^q([0,b], D(A_p))~\text{with}~\partial_t u\in L^q([0,b], H_p),
\end{equation}
\begin{equation}
    p\in L^q([0,b], W^{1,p}(D)^3)
\end{equation}
such that $(u,p)$ is a solution to the problem \eqref{Navier Stokes}-\eqref{IC} with $f\in L^q([0,b], L^p(D)^3)$ with $f(t)\in F(t, u(t))$ for a.a. $t\in [0,b]$.
\end{theorem}
\begin{proof}
For $b>0$, we introduce the space
    \begin{equation}
        U(b)=\{u\in L^q([0,b],D(A_p)): \partial_t u\in L^q([0,b],H_p)\}.
    \end{equation}
Consider the selection multimap $S_F: U(b)\subset L^q([0,b], H_p)\multimap L^q([0,b],L^p(D)^3)$ as follows:
 \begin{equation}
     S_F(u)=\{f\in L^q([0,b],L^p(D)^3): f(t)\in F(t,u(t)),~\text{a.a.}~t\in [0,b]\}.
 \end{equation}
 By virtue of assumption (F1), the multimap $F$ is product measurable. Hence, the multimap $t\multimap F(t,u(t))$ is closed valued and measurable for every $u\in L^q([0,b], H_p)$. In accordance with the Ryll-Kurtowski Selection Theorem (see \cite{MR1485775}), we can find a measurable mapping $f:[0,b]\to L^p(D)^3$ such that $f(t)\in F(t, u(t))$ for a.a. $t\in [0,b]$. By virtue of assumption (F3) we obtain
\begin{equation}
    \norm{f(t)}\le \alpha(t) (1+\eta_F(\norm{u(t)}_{H_p})), ~\text{a.a.}~t\in [0,b].
\end{equation}
We now estimate
\begin{equation}
\|f\|_{L^q([0,b],L^p(D)^3)}^q = \int_0^{b} \|f(t)\|_{L^p(D)^3}^q \, dt
\leq \int_0^{b} \alpha(t)^q \left(1 + \eta_F(\|u(t)\|_{H_p})\right)^q \, dt.
\end{equation}
Now, from the definition of the set $U(b)$, it is clear that if $u\in U(b)$, then $u$ is almost everywhere equal to an absolutely continuous function in $H_p$ and hence $u\in L^{\infty}([0,b], H_p)$.
Hölder’s inequality yields
\begin{equation}\label{4.9}
\|u\|_{L^\infty([0,b],H_p)} \leq \|u(0)\|_{H_p} + \|\partial_t u\|_{L^1([0,b],H_p)} \leq \|u(0)\|_{H_p} + b^{1/q^{\prime}} \|\partial_t u\|_{L^q([0,b],H_p)}.
\end{equation}
In the above, $q^{\prime}$ is the conjugate of $q$.

Together with \eqref{4.9} and with the monotonicity of \( \eta_F\), we have
\begin{equation}
\|f\|_{L^q([0,b],L^p(D)^3)}^q \leq \|\alpha\|_{L^q(0,b)}^q \left(1 + \eta_F\left(\|u(0)\|_{H_p} + b^{1/q^{\prime}} \|\partial_t u\|_{L^q([0,b],H_p)} \right)\right)^q.
\end{equation}
Consequently,
\begin{equation}\label{estimate 4}
  \|f\|_{L^q([0,b],L^p(D)^3)} \leq \|\alpha\|_{L^q(0,b)}  \left(1 + \eta_F\left(\|u(0)\|_{H_p} + b^{1/q^{\prime}} \|\partial_t u\|_{L^q([0,b],H_p)} \right)\right).
\end{equation}
From this we conclude that $f\in S_F(u)$ for every $u\in U(b)$. Consequently, the multimap $S_F$ is well defined. Also, $S_F$ is lower semicontinuous \cite[Theorem 3.2]{MR1972917} and has closed and decomposable values. Thus by Theorem \ref{decomposableselectiontheorem}  we obtain a continuous selection, say $\Lambda: U(b)\subset L^q([0,b], H_p)\to L^q([0,b],L^p(D)^3)$ from the multimap $S_F$, that means $\Lambda(u)\in S_F(u)$ for every $u\in U(b)$.

    For $R>0$, let
    \begin{equation}
        Y(b, R)=\{u\in U(b): u(0)=u_0, \norm{u}^q_{L^q([0,b],D(A_p))}+\norm{\partial_t u}^q_{L^q([0,b],H_p)}\le R^q\}.
    \end{equation}
    In view of Lemma \ref{main lemma}, we obtain a solution $u\in U(b)$ of the problem \eqref{Stokes equation} corresponding to $g=0$. Moreover, we have
     \begin{equation}
    \norm{u}^q_{L^q([0,b],D(A_p))}+\norm{\partial_tu}^q_{L^q([0,b],H_p)}\le c_A\left(\norm{u_0}^q_{D_p^q}\right).
     \end{equation}
     Therefore, if we choose $R^q>c_A(\norm{u_0}^q_{D_p^q})$, then we obtain 
      \begin{equation}
\norm{u}^q_{L^q([0,b],D(A_p))}+\norm{\partial_tu}^q_{L^q([0,b],H_p)}\le R^q.
     \end{equation}
     Therefore, $Y(b, R)$ is non-empty.

   We define $\Upsilon:Y(b,R)\to L^q([0,b], H_p)$ as follows: $\Upsilon(\tilde{u})=u$ if and only if $u\in L^q([0,b], H_p)$ is a solution to the following problem
    \begin{equation}\label{L1}
        \begin{cases}
            \partial_t u+\frac{1}{ \operatorname{Re}}A_pu=P_p(-(\tilde{u}\cdot \nabla)\tilde{u}+\Lambda \tilde{u}), ~\text{in}~(0,b)\\
            u(0)=u_0.
        \end{cases}
    \end{equation}
    In the above $P_p$ represents the Helmholtz projection operator.
    In view of Lemma \ref{main lemma} we obtain a unique solution $u\in U(b)$ to the problem \eqref{Stokes equation} corresponding to $g=P_p(-(\tilde{u}\cdot \nabla)\tilde{u}+\Lambda \tilde{u})\in L^q([0,b],H_p)$. Also, we obtain $u=\Upsilon(\tilde{u})$ satisfies the estimate
  \begin{equation}
\norm{u}^q_{L^q([0,b], D(A_p))}+\norm{\partial_tu}^q_{L^q([0,b], H_p)}\le c_A(\norm{u_0}^q_{D_p^q}+\norm{g}^q_{L^q([0,b], H_p)}), 
  \end{equation}
    where $c_A$ depends on $\operatorname{Re}$, $p,q$. Therefore, the map $\Upsilon$ is well defined. It is also clear that the fixed points of this map $\Upsilon$ are the solutions to the problem \eqref{Navier Stokes}. In fact,  let $u\in Y(b, R)$ be a fixed point of $\Upsilon$. Then, we have
    \begin{equation}
         \partial_t u+\frac{1}{\operatorname{Re} }A_pu=P_p(-(u\cdot \nabla)u+\Lambda u), ~\text{in}~ (0,b). 
    \end{equation}
    Since $\partial_tu\in L^q([0,b],H_p)$, we have $P_p(\partial_tu)=\partial_t u$. From the definition of the Stokes operator $A_p$, we obtain
    \begin{equation}
        P_p\left(\partial_tu-\frac{1}{\operatorname{Re}}\Delta u+(u\cdot \nabla)u-\Lambda(u))\right)=0,~\text{in}~(0,b).
    \end{equation}
    Now, the Helmholtz decomposition implies that, for almost all $t\in (0,b),$ there exists $p(t)\in W^{1,p}(D)^3$ such that
    \begin{equation}
    \partial_tu-\frac{1}{\operatorname{Re}}\Delta u+(u\cdot \nabla)u-\Lambda(u)=\nabla p,~\text{in}~(0,b).
    \end{equation}
    Due to the regularity of $u$ and the map $\Lambda$, we obtain $p\in L^q([0,b], W^{1,p}(D)^3)$. This shows that $(u,p)$ is a solution to the problem \eqref{Navier Stokes}.\\

    We now apply the Schauder Fixed Point Theorem in order to prove the existence of fixed points of the map $\Upsilon$. For this we need to prove that $\Upsilon$ maps the compact set $Y(b,R)\subset L^q([0,b], H_p)$ into $Y(b,R)$ for some $b, R>0$ and the map $\Upsilon$ is continuous.

    We start by proving that $\Upsilon$ maps $Y(b, R)$ into itself for some $b>0$ and $R>0$. Let $\tilde{u}\in Y(b, R).$ Then by means of Lemma \ref{estimate lemma}, we have $\tilde{u}\in L^{2q}([0,b],V_p)\cap L^{\frac{2pq}{3}}([0,b], L^{\infty}(D)^3)$ and there exists $c_1>0$, only depending on $p$ and $D$ such that
    \begin{equation}\label{estimate 1}
        \norm{\tilde{u}}_{L^{2q}([0,b],V_p)}\le c_1\norm{\tilde{u}}^{\frac{1}{2}}_{L^{\infty}([0,b],H_p)}\norm{\tilde{u}}^{\frac{1}{2}}_{L^q([0,b],D(A_p))}\le c_1\norm{\tilde{u}}^{\frac{1}{2}}_{L^{\infty}([0,b],H_p)}\norm{\tilde{u}}^{\frac{1}{2}}_{Y(b,R)}.
    \end{equation}
    Also, 
    \begin{equation}\label{estimate 2}
    \norm{\tilde{u}}_{L^{\frac{2pq}{3}}([0,b],L^{\infty}(D)^3)}\le c_1\norm{\tilde{u}}_{L^{\infty}([0,b],H_p)}\norm{\tilde{u}}^{\frac{3}{p}}_{L^{2q}([0,b],V_p)}\le c_1^{1+\frac{3}{p}}\norm{\tilde{u}}^{1-\frac{3}{2p}}_{L^{\infty}([0,b],H_p)}\norm{\tilde{u}}^{\frac{3}{2p}}_{Y(b,R)}.
    \end{equation}
    Now, Hölder’s inequality yields
    \begin{align*}
        \int_{0}^{b} \|\left(\tilde{u}(t) \cdot \nabla\right) \tilde{u}(t)\|^q_{L^p(D)^3} dt& \leq \int_{0}^{b} \|\tilde{u}(t)\|^q_{L^\infty(D)^3} \|\nabla \tilde{u}(t)\|^q_{L_p(D)^{3 \times 3}} dt\\
        \leq & \left( \int_{0}^{b} \|\tilde{u}(t)\|^{2q}_{L^\infty(D)^3} dt \right)^{1/2} \left( \int_{0}^{b} \|\nabla \tilde{u}(t)\|^{2q}_{L^p(D)^{3 \times 3}} dt \right)^{1/2}\\
        \leq& b^{\frac{p-3}{2p}} \left( \int_{0}^{b} \|\tilde{u}(t)\|^{2pq/3}_{L^\infty(D)^3} dt \right)^{3/(2p)} \left( \int_{0}^{b} \|\nabla \tilde{u}(t)\|^{2q}_{L^p(D)^{3 \times 3}} dt \right)^{1/2}\\
        =& b^{\frac{p-3}{2p}} \|\tilde{u}\|^q_{L^{2pq/3}([0,b],L^\infty(D)^3)} \|\tilde{u}\|^q_{L^{2q}([0,b],V_p)}.
    \end{align*}
Therefore, using estimates \eqref{estimate 1} and \eqref{estimate 2} we obtain 
 \begin{align*}
        \left(\int_{0}^{b} \|\left(\tilde{u}(t) \cdot \nabla\right) \tilde{u}(t)\|^q_{L^p(D)^3} dt\right)^{\frac{1}{q}}
        \le &  b^{\frac{p-3}{2pq}}c_1^{2+\frac{3}{p}}\norm{\tilde{u}}^{1-\frac{3}{2p}}_{L^{\infty}([0,b],H_p)}\norm{\tilde{u}}^{\frac{3}{2p}}_{Y(b,R)}  \norm{\tilde{u}}^{\frac{1}{2}}_{L^{\infty}([0,b],H_p)}\norm{\tilde{u}}^{\frac{1}{2}}_{Y(b,R)}.
    \end{align*}
Finally, the continuous embedding of $Y(b,R)$ into $L^{\infty}([0,b],H_p)$ implies that there exists $c_2>0$, such that 
\begin{equation}
    \norm{v}_{L^{\infty}([0,b]H_p)}\le c_2\norm{v}_{Y(b,R)}, ~\text{for all}~v\in Y(b,R);
\end{equation}
and thus
\begin{align*}
        \left(\int_{0}^{b} \|\left(\tilde{u}(t) \cdot \nabla\right) \tilde{u}(t)\|^q_{L^p(D)^3} dt\right)^{\frac{1}{q}}
        \le &  b^{\frac{p-3}{2pq}}c_1^{2+\frac{3}{p}} c_2^{\frac{3}{2}-\frac{3}{2p}}\norm{\tilde{u}}^2_{Y(b,R)}\le  b^{\frac{p-3}{2pq}}c_1^{2+\frac{3}{p}} c_2^{\frac{3}{2}-\frac{3}{2p}}R^2.
    \end{align*}
Therefore,
\begin{equation}\label{estimate 3}
     \norm{P_p(-(\tilde{u}\cdot \nabla)\tilde{u})}_{L^q([0,b],H_p)} \leq  b^{\frac{p-3}{2pq}}c_1^{2+\frac{3}{p}} c_2^{\frac{3}{2}-\frac{3}{2p}}R^2.
\end{equation}
Now, noticing that $\Lambda$ is a continuous selection of the multimap $S_F$, we obtain
\begin{equation}
(\Lambda \tilde{u})(t)\in F(t, \tilde{u}(t))~\text{for a.a.}~t\in [0,b].
\end{equation}
By virtue of assumption (F3), we obtain
\begin{equation}
    \norm{(\Lambda \tilde{u})(t)}\le \alpha(t) (1+\eta_F(\norm{\tilde{u(t)}}_{H_p}), ~\text{a.a.}~t\in [0,b].
\end{equation}
We now estimate
\begin{equation}\label{4.99}
\|\Lambda (\tilde{u})\|_{L^q([0,b],L^p(D)^3)}^q = \int_0^{b} \|\Lambda(\tilde{u})(t)\|_{L^p(D)^3}^q \, dt
\leq \int_0^{b} \alpha(t)^q \left(1 + \eta_F(\|\tilde{u}(t)\|_{H_p})\right)^q \, dt.
\end{equation}
As mentioned before, we have $\tilde{u}: [0,b]\to H_p$ is almost everywhere equal to an absolutely continuous function and hence \( u \in L^{\infty}([0,b],H_p) \). Hölder’s inequality yields
\begin{equation}
\|\tilde{u}\|_{L^\infty([0,b],H_p)} \leq \|u_0\|_{H_p} + \|\partial_t \tilde{u}\|_{L^1([0,b],H_p)} \leq \|u_0\|_{H_p} + b^{1/q^{\prime}} \|\partial_t \tilde{u}\|_{L^q([0,b],H_p)}.
\end{equation}
In the above, $q^{\prime}$ is the conjugate of $q$.

Together with \eqref{4.99} and with the monotonicity of \( \eta_F\), we have
\begin{equation}
\|\Lambda (\tilde{u})\|_{L^q([0,b],L^p(D)^3)}^q \leq \|\alpha\|_{L^q(0,b)}^q \left(1 + \eta_F\left(\|u_0\|_{H_p} + b^{1/q^{\prime}} \|\partial_t \tilde{u}\|_{L^q([0,b],H_p)} \right)\right)^q,
\end{equation}
so \( \tilde{u} \in Y(b,R) \) implies
\begin{equation}
\|\Lambda (\tilde{u})\|_{L^q([0,b],L^p(D)^3)}^q \leq \|\alpha\|_{L^q(0,b)}^q \left(1 + \eta_F\left(\|u_0\|_{H_p} + R b^{1/q^{\prime}} \right)\right)^q.
\end{equation}
Consequently,
\begin{equation}\label{estimate 4}
  \|\Lambda (\tilde{u})\|_{L^q([0,b],L^p(D)^3)} \leq \|\alpha\|_{L^q([0,b])} \left(1 + \eta_F\left(\|u_0\|_{H_p} + R b^{1/q^{\prime}} \right)\right).  
\end{equation}
Therefore, combining estimates \eqref{estimate 3} and \eqref{estimate 4} we obtain
  \begin{align*}
\norm{\tilde{u}}^q_{L^q([0,b], D(A_p))}+\norm{\partial_t\tilde{u}}^q_{L^q([0,b], H_p)}\le& c_A\left(\norm{u_0}^q_{D_p^q}+ b^{\frac{p-3}{2p}}c_1^{q(2+\frac{3}{p})} c_2^{q(\frac{3}{2}-\frac{3}{2p})}R^{2q}\right.\\
&\left. \hspace{3cm}+\|\alpha\|_{L^q[0,b]}^q \left(1 + \eta_F\left(\|u_0\|_{H_p} + R b^{1/q^{\prime}} \right)\right)^q\right)
  \end{align*}
Following the paper \cite{MR4478379} we can now choose $R_0>0$ and $b_0>0$ small such that we have
 \begin{align*}
\norm{u}^q_{L^q([0,b], D(A_p))}+\norm{\partial_tu}^q_{L^q([0,b], H_p)}\le&  R_0^q.
  \end{align*}
  This concludes that $\Upsilon$ maps $Y(b_0,R_0)$ into itself.

  We now prove that the map $\Upsilon$ is continuous. For this we consider $\tilde{u}_n\in Y(b_0, R_0)$ such that $\tilde{u}_n\to \tilde{u}$ in $L^q([0,b], H_p)$. We prove $\Upsilon(\tilde{u}_n)\to \Upsilon(\tilde{u})$ in $L^q([0,b_0], H_p)$. Let $u_n=\Upsilon(\tilde{u}_n)$ and $u=\Upsilon(\tilde{u})$. By the compact embedding of $D(A_p)$ into $H_p$, we obtain $u_n\to v$ in $L^q([0,b_0], H_p)$ for some $v\in L^q([0,b_0],H_p)$. Also, $\{u_n\}_{n\in \mathbb{N}}$ is bounded in $L^q([0,b_0], D(A_p))$ and $\{u_n^{\prime}\}_{n\in \mathbb{N}}$ is bounded in $L^q([0,b_0], H_p)$. By the reflexivity of the Banach spaces $L^q([0,b_0], D(A_p))$ and $L^q([0,b_0], H_p)$ we obtain $u_n \rightharpoonup u$ in $L^q([0,b],D(A_p))$ and $\partial_t u_n\rightharpoonup w$ in $L^q([0,b],H_p)$. It is clear that $w=\partial_t u$. By the uniqueness of the weak limit, we obtain $u=v$. Therefore, $u_n\to u$ in $L^q([0,b_0], H_p)$. By the definition of the map $\Upsilon$, we obtain $u_n$ as the unique solution to the problem
         \begin{equation}
        \begin{cases}
            \operatorname{Re} \partial_t u_n+(1-\alpha)A_pu_n=P_p(-\operatorname{Re}(\tilde{u}_n\cdot \nabla)\tilde{u}_n+\Lambda(\tilde{u_n})), ~\text{in}~(0,b)\\
            u_n(0)=u_0.
        \end{cases}
    \end{equation}
The operator $A_p: D(A_p)\to H_p$ is a linear, bounded operator and weakly sequentially continuous. Also, the map $\Lambda$ is continuous. Therefore, following the proof of \cite[Theorem 4.1]{MR4478379} and passing the limit, we obtain $u$ satisfies
         \begin{equation}
        \begin{cases}
            \operatorname{Re} \partial_t u+(1-\alpha)A_pu=P_p(-\operatorname{Re}(\tilde{u}\cdot \nabla)\tilde{u}++\Lambda(u)), ~\text{in}~(0,b)\\
            u(0)=u_0.
        \end{cases}
    \end{equation}
    This concludes that the map $\Upsilon$ is continuous. Also, it is clear that the set $Y(b_0,R_0)$ is compact in $L^q([0,b],H_p)$ (see the book \cite[Chapter II]{MR1893419}). Therefore, by virtue of Schauder's fixed point theorem, we conclude that the map $\Upsilon$ has a fixed point $u\in Y(b_0, R_0)$ (say). By the definition of the map $\Upsilon$, we say that $u$ is a solution to the problem \eqref{Navier Stokes}-\eqref{IC}.

\end{proof}

\section{Main Results}
In this section, we generalize the main result obtained in the previous section of this paper. 
Let $b>0$, let $X, Z$ be real, reflexive and separable Banach spaces, and let $Y$ be a real Banach space such that we have the continuous embeddings
\begin{equation}
    X\hookrightarrow Y \hookrightarrow Z,
\end{equation}
and the compact embedding $X\hookrightarrow Z$. Further, let $b\in (0,a], ~q\in (1,\infty)$ and let
\begin{equation}
    X(b)=\{u\in L^q([0,b],X): u^{\prime}~\text{exists and}~u^{\prime}\in L^q([0,b],Z)\}.
\end{equation}
Then $X(b)$ is a Banach space endowed with the norm
\begin{equation}
\norm{u}_{X(b)}=\left(\norm{u}^q_{L^q([0,b],X)}+\norm{u^{\prime}}^q_{L^q([0,b],Z)}\right)^{\frac{1}{q}}.
\end{equation}
Before proceeding further, let us define the solution concept of the problem \eqref{Inclusion Problem 1}. We define the measurable selection multimap $S_F:L^q([0,b],X)\multimap L^q([0,b],Z)$ as follows:
\begin{equation}\label{selection multimap}
    S_F(u)=\{f\in L^q([0,b],Z): f(t)\in F(t,u(t))~~\text{a.a.}~t\in [0,b]\}. 
\end{equation}
\begin{definition}
    We say that a function $u\in X(b)$ is a solution to the problem \eqref{Inclusion Problem 1} if $u$ satisfies
    \begin{equation}
          \begin{cases}
        u^{\prime}(t)+A(t)u(t)+B(t)u(t)=f(t),~\text{in}~Z,~ t\in (0,b).\\
        u(0)=u_0~\text{in}~Z,
    \end{cases}
    \end{equation}
    for some $f\in S_F(u)$.
\end{definition}
We now define the operators $\mathcal{A}:X(b)\to L^q([0,b],Z)$ defined by $(\mathcal{A}v)(t)=A(t)v(t)$  for almost all $t\in (0,b)$, and the operator $\mathcal{B}:X(b)\subset L^q([0,b],Z)\to L^q([0,b],Z)$ defined by
    \begin{equation}
        (\mathcal{B}v)(t)=B(t)v(t),~\text{for almost all}~t\in (0,b),
    \end{equation}
We now assume the following assumptions on the operators $\mathcal{A}$ and $\mathcal{B}$ and the multimap $F$.
\begin{itemize}
    \item[(A1)] the operator $\mathcal{A}$ is well defined, linear and bounded.
    \item[(A2)] for $u_0\in Y$ and $g\in L^q([0,b],Z)$ there exists  unique solution $u\in X(b)$ to the problem
    \begin{equation}\label{Linear problem 1}
        \begin{cases}
            u^{\prime}+\mathcal{A}u=g, ~\text{in}~L^q([0,b],Z)\\
            u(0)=u_0,~\text{in}~Z.
        \end{cases}
    \end{equation}
    \item[(A3)] there exists $c_A>0$ such that all solutions $u\in X(b)$ to problem \eqref{Linear problem 1} fulfil the a priori estimate 
    \begin{equation}
        \norm{u}_{X(b)}\le c_A(\norm{u_0}_Y+\norm{g}_{L^q([0,b],Z)}).
    \end{equation}
    \item[(B1)] The operator $\mathcal{B}$ is well defined, and sequentially continuous.
    \item[(B2)] for all $u_0\in Y$, there exists $c_{B,u_0}>0$ and a function $\sigma_{u_0}: [0,\infty)\times (0,b]\to [0,\infty)$ that is monotonically increasing in both arguments such that 
    \begin{equation}
        \norm{\mathcal{B}u}_{L^q([0,b],Z)}\le \sigma_{u_0}(\norm{u}_{X(b)},b)
    \end{equation}
    for almost all $u\in X(b)$ with $u(0)=u_0$ in $Z$ and such that for every $R>0,$ there exists $b_R\in (0,b]$ with $\sigma_{u_0}(R,b_R)\le c_{B,u_0}.$
    \item[(F1)] $F: [0,a]\times Z\multimap Z$ is product measurable.
    \item[(F2)] $F(t,\cdot): Z\multimap Z$ is lower semicontinuous for a.a. $t\in [0,a]$.
    \item[(F3)] there exists $\alpha\in L^q(0,a)$ with $\alpha\ge 0$ for a.a. $ t\in (0,a)$ and a monotonically increasing function $\eta_F:[0,\infty)\to [0,\infty)$ such that
    \begin{equation}
        \norm{F(t,u)}\le \alpha(t)(1+\eta_F(\norm{u}_Z)),
    \end{equation}
    for a.a. $t\in (0,a)$ and all $u\in Z$.
\end{itemize}
We now establish the second main result of this paper.
\begin{theorem}\label{Main result 2}
Let the assumptions (A1)-(A3), (B1)-(B2) and (F1)-(F3) holds. Then, if $u_0\in Y$, there exists a local solution to problem \eqref{Inclusion Problem 1}, that is there exists $b_0\in (0,a], u\in X(b_0)$ and $f\in S_F(u)$ such that
\begin{equation}\label{Main Inclusion Problem}
    \begin{cases}
        u^{\prime}+\mathcal{A}u+\mathcal{B}u=f, ~\text{in}~L^q(0,b_0;Z)\\
        u(0)=u_0, ~\text{in}~Z.
    \end{cases}
\end{equation}
\end{theorem}
\begin{proof}
 We consider the set
 \begin{equation}
     Y(b,R)=\{u\in X(b): u(0)=u_0, ~\text{in}~Z, \norm{u}_{X(b)}\le R\},
 \end{equation}
 for arbitrary $R>0$. We have continuous embeddings 
 \begin{equation}\label{continuous embedding}
     X(b)\hookrightarrow W^{1,1}([0,b],Z)\hookrightarrow C([0,b];Z),
 \end{equation}
 so functions in $X(b)$ are almost everywhere equal to functions on $[0,b]$ with values in $Z$ that are continuous. Therefore, $Y(b, R)$ is well-defined. First, we show that $Y(b, R)$ is non-empty for large enough $R$. Due to assumptions (A2) and (A3), there exists a solution $\tilde{u}\in X(b)$ to the linear problem \eqref{Linear problem 1} with right-hand side $g=0$, and this solution fulfils the bound
 \begin{equation}
     \norm{\tilde{u}}_{X(b)}\le c_A\norm{u_0}_Y,
 \end{equation}
 so if we choose $R\ge c_A\norm{u_0}_Y$, we have $\tilde{u}\in Y(b,R)$, that is, $Y(b,R)$ is nonempty. 
 
 We now consider the selection multimap $S_F: X(b)\subset L^q([0,b],Z)\multimap L^q([0,b],Z)$ as in \eqref{selection multimap}.
 By virtue of assumption (F1), the multimap $F$ is product measurable. Hence, the multimap $t\multimap F(t,u(t))$ is closed valued and measurable for every $u\in X(b)$. In accordance with the Ryll-Kurtowski Selection Theorem, we can find a measurable mapping $f:[0,b]\to Z$ such that $f(t)\in F(t, u(t))$ for a.a. $t\in [0,b]$. Hence $f\in S_F(u)$ for every $u\in X(b)$. Consequently, the multimap $S_F$ is well defined. As in the proof of Theorem \ref{Main Theorem} we obtain a continuous selection, say $\Lambda: X(b)\to L^q([0,b],Z)$ from the multimap $S_F$, that means $\Lambda(u)\in S_F(u)$ for every $u\in X(b)$.

 Now, we define $\Upsilon: Y(b,R)\to L^q([0,b],Z)$ with $u=\Upsilon(\tilde{u})$ if and only if $u\in X(b)$ is a solution to the problem
 \begin{equation}\label{P2}
     \begin{cases}
         u^{\prime}(t)+(\mathcal{A}u)(t)=-(\mathcal{B}\tilde{u})(t)+(\Lambda \tilde{u})(t),~\text{in}~Z\\
         u(0)=u_0,~\text{in}~Z.
     \end{cases}
 \end{equation}
 By virtue of assumption (A2), the problem \eqref{P2} has unique solution $u\in X(b)$ satisfying
   \begin{equation}
        \norm{u}_{X(b)}\le c_A(\norm{u_0}_Y+\norm{g}_{L^q([0,b],Z)}),
    \end{equation}
    where $g=-\mathcal{B}\tilde{u}+\Lambda \tilde{u}\in L^q([0,b],Z)$. Therefore, the map $\Upsilon$ is well defined, and a fixed point of $\Upsilon$ is a solution to problem \eqref{Main Inclusion Problem}. Consequently, our aim is to prove the existence of fixed points of the map $\Upsilon$. We employ the Schauder Fixed Point Theorem to fulfill this claim. In order to do so, we need to verify that $\Upsilon$ maps $Y(b, R)$ into itself, $Y(b, R)\subset L^q([0,b], Z)$ is compact, and $\Upsilon$ is continuous. 
    
    Let $\tilde{u}\in Y(b,R)$ and $u=\Upsilon(\tilde{u})$. Due to assumption (A3), we have the estimate
    \begin{equation}\label{expression 1}
        \norm{u}_{X(b)}\le c_A\left(\norm{u_0}_Y+\norm{\mathcal{B}\tilde{u}}_{L^q([0,b],Z)}+\left(\int_{0}^{b}\norm{\Lambda\tilde{u}(t)}^q_Zdt\right)^{\frac{1}{q}}\right).
    \end{equation}
    We now estimate $\norm{\tilde{u}(t)}_Z$. From the continuous embedding $X(b)\hookrightarrow W^{1,1}([0,b],Z)$, $\tilde{u}$ is almost everywhere equal to an absolutely continuous function on $[0,b]$ with values in $Z$. Therefore, using the fundamental theorem of calculus for absolutely continuous functions, we obtain
\begin{align*}
    \norm{\tilde{u}(t)}_Z\le& \norm{u_0}_Z+\int_{0}^{t}\norm{\tilde{u}^{\prime}(\tau)}_Zd\tau ~\text{for all}~t\in [0,b].
\end{align*}
In view of Holder's inequality, we obtain
\begin{align*}
    \norm{\tilde{u}(t)}_Z\le \norm{u_0}_Z+b^{\frac{q-1}{q}}\norm{\tilde{u}^{\prime}}_{L^q([0,b],Z)}
    \le \norm{u_0}_Z+b^{\frac{q-1}{q}}R, ~\text{for all}~t\in [0,b].
\end{align*}
By virtue of assumption (B2), using the monotonicity of the map $\sigma_{u_0}$ we obtain
\begin{equation}\label{expression 2}
    \norm{\mathcal{B}\tilde{u}}_{L^q([0,b],Z)}\le \sigma_{u_0}(\norm{\tilde{u}}_{X(b)},b)\le \sigma_{u_0}(R, b).
\end{equation}
Now, the growth condition on $F$ and the monotonicity of $\eta_F$, (see assumption (F3)) imply
\begin{equation}
   \norm{F(t,\tilde{u}(t)}_Z\le \alpha(t)(1+\eta_F(\norm{\tilde{u}(t)}_Z)\le \alpha(t)(1+\eta_F(\norm{u_0}_Z+b^{\frac{q-1}{q}}R)) 
\end{equation}
for almost all $t\in (0,b)$ and therefore
\begin{align}\label{expression 3}
    \left(\int_{0}^{b}\norm{(\Lambda \tilde{u})(t)}_Z^qdt\right)^{\frac{1}{q}} \le &\left(\int_{0}^{b}\alpha(t)^q\left(1+\eta_F(\norm{u_0}_Z+b^{\frac{q-1}{q}}R)\right)^qdt\right)^{\frac{1}{q}}\\
    =& \norm{\alpha}_{L^q(0,b)}\left(1+\eta_F(\norm{u_0}_Z+b^{\frac{q-1}{q}}R)\right).
\end{align}
Combining expressions \eqref{expression 2} and \eqref{expression 3} we obtain from \eqref{expression 1},
\begin{equation}
    \norm{u}_{X(b)}\le c_A\left(\norm{u_0}_Y+\sigma_{u_0}(R,b)+\norm{\alpha}_{L^q(0,b)}(1+\eta_F(\norm{u_0}_Z+b^{\frac{q-1}{q}}R\right).
\end{equation}
Now, by choosing $R_0=c_A(\norm{u_0}_Y+\sigma_{u_0}(R,b)+\norm{\alpha}_{L^q(0,b)}(1+\eta_F(1+\norm{u_0}_Z)))$, and following \cite{{MR1893419}}, we can prove that the map $\Upsilon$ maps $Y(b_0, R_0)$ into itself, for some $b_0>0$.

The compactness of the set $Y(b_0,R_0)$ follows from the compact embedding of $X(b_0)$ into $L^q([0,b],Z)$. Therefore, it remains to prove that the map $\Upsilon$ is continuous. For this we consider $\tilde{u}_n\in Y(b_0, R_0)$ such that $\tilde{u_n}\to \tilde{u}$ in $L^q([0,b_0],Z)$. We prove that $\Upsilon(\tilde{u}_n)\to \Upsilon(\tilde{u})$ in $L^q([0,b_0],Z)$. Let $u_n=\Upsilon(\tilde{u}_n)$ and $u=\Upsilon(\tilde{u})$. By definition, $u_n$ is the solution to the problem
 \begin{equation}\label{P3}
     \begin{cases}
         u_n^{\prime}(t)+(\mathcal{A}u_n)(t)=-(\mathcal{B}\tilde{u}_n)(t)+(\Lambda \tilde{u}_n)(t),~\text{in}~Z\\
         u_n(0)=u_0,~\text{in}~Z.
     \end{cases}
 \end{equation}
 As $u_n\in Y(b_0,R_0)$, we have $\{u_n\}_{n\in \mathbb{N}}$ is bounded in $X(b_0).$ As the embedding from $X(b_0)$ into $L^q([0,b_0],Z)$ is compact, we obtain $u_n\to v$ in $L^q([0,b_0],Z)$ up to a subsequence. Also, the sequence $\{u_n\}_{n\in \mathbb{N}}$ is bounded in $X(b_0)$. By reflexivity of the spaces $L^q([0,b_0],X)$ and $L^q([0,b_0],Z)$ we conclude that $u_n\rightharpoonup u$ in $L^q([0,b_0],X)$ and $u_n^{\prime}\rightharpoonup u^{\prime}$ in $L^q([0,b_0],Z)$. From equation \eqref{P3} we obtain
 \begin{equation}
      u_n^{\prime}=-(\mathcal{A}u_n)-(\mathcal{B}\tilde{u}_n)+(\Lambda \tilde{u}_n)\rightharpoonup -(\mathcal{A}u)-(\mathcal{B}\tilde{u})+(\Lambda \tilde{u}), ~\text{in}~L^q([0,b],Z). 
 \end{equation}
 Therefore, by the uniqueness of limit we obtain $v=u$ and
 \begin{equation}
     u^{\prime}= -(\mathcal{A}u)-(\mathcal{B}\tilde{u})+(\Lambda \tilde{u}).
 \end{equation} 
 This shows that $u=\Upsilon(\tilde{u})$ and $\Upsilon(\tilde{u}_n)=u_n\to u=\Upsilon(\tilde{u})$. Thus, we have proved that the map $\Upsilon$ is continuous. By Schauder's fixed point theorem, we conclude that the map $\Upsilon$ has a fixed point, say $u$ and $u$ is a solution to the problem
 \begin{equation}
     \begin{cases}
         u^{\prime}(t)+(\mathcal{A}u)(t)=-(\mathcal{B}u)(t)+(\Lambda u)(t),~\text{in}~Z\\
         u(0)=u_0,~\text{in}~Z.
     \end{cases}
 \end{equation}
 Now, by the definition of the map $\Lambda$, we obtain $(\Lambda u)(t)\in F(t,u(t))$; hence $u\in X(b_0)$ is a solution to the problem \eqref{Main Inclusion Problem}. 
\end{proof}

\begin{remark}
The problem \eqref{Navier Stokes} can be converted into the inclusion problem \eqref{Inclusion Problem 1} by defining the operators $\mathcal{A}$ and $\mathcal{B}$ suitably. In fact, we define $\mathcal{A}: L^q([0,b], D(A_p))\to L^q([0,b], H_p)$ as follows:
\begin{equation}
    (\mathcal{A}u)(t)=\frac{1}{\operatorname{Re}}A(t)u(t),~\text{for a.a.}~t\in [0,b].
\end{equation}
Further, define $\mathcal{B}: X(b)\to L^q([0,b], H_p)$ as $(\mathcal{B}u)(t)=P_p(\tilde{B}u)(t)$ for a.a. $t\in [0,b]$, where $\tilde{B}: X(b)\to L^q([0,b], L^p(D)^3)$ is given by
\begin{equation}
    (\tilde{B}(u))(t)=(u(t)\cdot \nabla)u(t),~\text{for a.a.}~t\in [0,b].
\end{equation}
Then, following \cite{eikmeier2023existence}, we can see that the operators $\mathcal{A}$ and $\mathcal{B}$ satisfy all the assumptions required to prove Theorem \ref{Main result 2}. 
\end{remark}

				\bibliographystyle{elsarticle-num}
				\bibliography{sn-bibliography}
				
			\end{document}